\documentclass{sl}    %% This the CLASS for Studia Logica,
\usepackage{slsec}     %% PACKAGES for Studia Logica
\usepackage{stud_log} 
\usepackage{slfoot}
\usepackage{slthm}     %% Theorem-like definitions as in amsthm,
\usepackage{amsmath}
\usepackage{amssymb}

\usepackage{enumerate}
\def\kn{\kern.1em}

\renewcommand{\theequation}{\Roman{equation}}
\newtheorem{theo}{Theorem}[section]  
\newtheorem{coro}[theo]{Corollary}

\theoremstyle{definition}
\newtheorem{defi}[theo]{Definition}

\theoremstyle{proposition}
\newtheorem{prop}[theo]{Proposition}

\theoremstyle{lemma}
\newtheorem{lem}[theo]{Lemma}

\usepackage{latexsym}

\begin{document}

\title{\bf On Blass translation for Le\'{s}niewski's propositional ontology and modal logics}
\author{\bf Takao Inou\'{e}} 
\date{}
\maketitle

\setcounter{page}{1}     %%  Initial page number. Do not change. 

%%%%%%  FORMAT OF AUTHOR AND TITLE INFORMATION:
 
%%%%%%%%%%%%%%%%%%% The (default) case of one author
%%%%%%%%%%%%%%%%%%% and one or two lines of title

%\AuthorTitle{Takao Inou\'{e}}{On Blass translation for Le\'{s}niewski's propositional ontology and modal logics}
%\newline and modal logic with irreflexive frames}

%%%%%%%%%%%%%%%%%%% The case of two authors
%%%%%%%%%%%%%%%%%%% and one line title
%\twoAuthorsTitleoneline {J.\ts Smith}{W.\ts A. Novak}{The One 
%Line Title}

%%%%%%%%%%%%%%%%%%% The case of two authors
%%%%%%%%%%%%%%%%%%% and at least two lines title
%\twoAuthorsTitle{M.\ts Smith}{W.\ts A. Novak}{The First 
%Line of The Title\linebreak 
%and the Second Line}

%%%%%%%%%%%%%%%%%%% The case of 
%%%%%%%%%%%%%%%%%%% three authors & one or two lines title 
%\threeAuthorsTitle{M.\ts Smith}{W.\ts A. Novak}{C. 
%\ts Johns}{The first 
%line of the title
%\linebreak and the second line}

%%%%%%%%%%%%%%%%%%% The case of 
%%%%%%%%%%%%%%%%%%% three authors & at least three lines title 
%\threeAuthorsTitlethreelines{M.\ts Smith}{W.\ts A. Novak}
%{C. \ts Johns}{The First Line of The Title\linebreak 
%The Second Line of The Title,\linebreak
%and The Third Line}

%%%%%%  INFORMATION FOR FOOTER OF FIRST PAGE
%%%%%%  will be inserted by the Editorial Office.
   
%\PresentedReceived{Heinrich Wansing}{June 2020}

%%%%%%  ABSTRACT AND KEYWORDS  (obligatory)

%%%%%%  ABSTRACT AND KEYWORDS  (obligatory)

\centerline{This is the second revised version (v2) of this paper, October 31, 2020.}

\bigskip

\begin{abstract}
In this paper, we shall give another proof of the faithfulness of Blass translation (for short, $B$-translation) of the propositional 
fragment $\bf L_1$ of Le\'{s}niewski's ontology in the modal logic 
$\bf K$ \it by means of Hintikka formula\rm . And we extend the result to  
von Wright-type deontic logics, i.e., ten Smiley-Hanson systems of monadic deontic logic.
As a result of observing the proofs we shall give general theorems on the faithfulness of 
$B$-translation with respect to normal modal logics complete to certain sets of well-known 
accessibility relations with a restriction that transitivity and symmetry are not set at 
the same time. As an application of the theorems, for example, $B$-translation is faithful 
for the provability logic $\bf PrL$ (= $\bf GL$), that is, 
$\bf K$ $+$ $\Box (\Box \phi \supset \phi) \supset \Box \phi$. The faithfulness also holds 
for normal modal logics, e.g., $\bf KD$, $\bf K4$, $\bf KD4$, $\bf KB$. We shall conclude 
this paper with the section of some open problems and conjectures.
\end{abstract}


%%%%%  THE BODY OF THE PAPER

\section{Introduction}

In Inou\'{e} \cite{inoue1}, a partial interpretation of Le\'{s}niewski's epsion $\epsilon$ in the modal logic $\bf K$ and 
its certain extensions was proposed: that is, Ishimoto's propositional fragment $\bf L_1$ (Ishimoto \cite{ishi}) of 
Le\'{s}niewski's ontology $\bf L$ (refer to Urbaniak \cite{urbaniak-book}) is partially embedded in $\bf K$ and 
in the extensions, respectively, by the following 
translation $I(\cdot)$ from $\bf L_1$ to them:

\smallskip

(1.i) \enspace $I(\phi \vee \psi)$ = $I(\phi) \vee I(\psi)$, 

\smallskip

(1.ii) \enspace $I(\neg \phi)$ = $\neg I(\phi)$, 

\smallskip

(1.iii) \enspace $I(\epsilon ab)$ = $p_a \wedge \Box (p_a \equiv p_b)$,

\smallskip

\noindent where $p_a$ and $p_b$ are propositional variables corresponding to the name variables $a$ and $b$, respectively. 
Here, ``$\bf L_1$ is partially embedded in $\bf K$ by $I(\cdot)$" means that for any formula $\phi$ 
of a certain decidable nonempty set of formulas of $\bf L_1$ (i.e. decent formulas (see $\S 3$ of Inou\'{e} \cite{inoue3})), $\phi$ is 
a theorem of $\bf L_1$ if and only if $I(\phi)$ is a theorem of $\bf K$. Note that $I(\cdot)$ is sound. 

The paper \cite{inoue3} also proposed similar interpretations of Le\'{s}niewski's epsilon in certain 
von Wright-type deontic logics, that is, ten Smiley-Hanson systems of monadic deontic logic
and in provability logics (i.e., the full system $\bf PrL$ (= $\bf GL$) of provability logic and its subsystem $\bf BML$ (= $\bf K4$), 
respectively. (See \AA qvist \cite{aq}, and Smory\'{n}ski \cite{smo} or Boolos \cite{boolos} for those logics.)

The interpretation $I(\cdot)$ is however not faithful. A counterexample for the faithfulness is, for example, 
$\epsilon ac \wedge \epsilon bc. \supset . \epsilon ab \vee \epsilon cc$ (for the details, see \cite{inoue3}). Blass \cite{blass} 
gave a modification of the interpretation and showed that his interpretation $T(\cdot)$ is faithful, using Kripke 
models. In this paper, we shall call the faithful interpretation (denoted by $B(\cdot)$) \it Blass 
translation \rm (for short, \it $B$-translation\rm ) or  \it Blass interpretation \rm (for short, \it B-interpretation\rm ). 
The translation $B(\cdot)$ from $\bf L_1$ to $\bf K$ is defined as follows:

\smallskip

(2.i) \enspace $B(\phi \vee \psi)$ = $B(\phi) \vee B(\psi)$, 

\smallskip

(2.ii) \enspace $B(\neg \phi)$ = $\neg B(\phi)$, 

\smallskip

(2.iii) \enspace $B(\epsilon ab)$ = $p_a \wedge \Box (p_a \supset p_b) \wedge. p_b \supset \Box (p_b \supset p_a)$,

\smallskip

\noindent where $p_a$ and $p_b$ are propositional variables corresponding to the name variables $a$ and $b$, respectively.

Now the purpose of this paper is to give another proof of the faithfulness of Blass translation 
($B$-translation) of Le\'{s}niewski's propositional ontology $\bf L_1$ in the modal logic 
$\bf K$ \it by means of Hintikka formula\rm . This will be done in \S4. In \S 3, we shall give chain equivalence 
relation and so on, as important technical tools for this paper. This section wil present the notion of chain and tail 
in Kobayashi and Ishimoto \cite{koishi} as a refined version of it. After the 
main result in \S4, we shall extend the faithfulness result to von Wright-type deontic logics, i.e., 
ten Smiley-Hanson systems of monadic deontic logics in \S5. Observing the proof of the faithfulness 
for $\bf K$,  we shall, in \S6, give a general theorem on the faithfulness of $B$-translation with 
respect to normal modal logics with transitive or irreflexive frames. As an application of the generalization, 
we shall obtain the faithfulness for provability logics (i.e., the full system $\bf PrL$ (= $\bf GL$) of 
provability logic and its subsystem $\bf BML$ (= $\bf K4$)). $\bf PrL$ is characterized by 
$\bf K$ $+$ $\Box (\Box \phi \supset \phi) \supset \Box \phi$. $\bf PrL$ is also characterized by 
$\bf K4$ $+$ $\Box (\Box \phi \supset \phi) \supset \Box \phi$. From careful consideration of 
the proof of the faithfulness for deontic logics, we shall obtain further general theorems for 
the faithfulness for the propositional normal modal logics which are complete to certain sets of well-known
accessibility relations with a restriction that transitivity and symmetry are not set at 
the same time. As the result of the consideration, we shall show that the faithfulness also holds for normal modal 
logics, e.g., $\bf KD$, $\bf KT$, $\bf K4$, $\bf KD4$ and $\bf KB$. 
This will be presented in \S7. The last section, \S8 contains some open problems and conjectures. 
In the following \S2, we shall first collect the basic preliminaries for this paper. 

Before ending of this section, we shall give some comments over my motivations and the significance of this paper.

One motive from which I wrote \cite{inoue1} and \cite{inoue3} is that I wished to understand Le\'{s}niewski's epsilon $\epsilon$ on the basis of 
my recognition that Le\'{s}niewski's epsilon would be a variant of truth-functional equivalence $\equiv$. Namely, my original approach to the 
interpretation of $\epsilon$ was to express the deflection of $\epsilon$ from $\equiv$ in terms of Kripke models. \it Here we emphasize 
to mention that this recognition for Le\'{s}niewski's epsilon is in effect within the propositional fragment $\bf L_1$ of 
Le\'{s}niewski's ontology $\bf L$. We do not intend to mean it for the whole system $\bf L$\rm . We may say that this is an 
approximation of $\epsilon$ by something simple and symmetric, that is equivalence. In analysis, the fundamental idea of calculus is the local 
approximation of functions by linear functions. In mathematics, we can find a lot of such examples as that of calculus. As an example in logic, 
a Gentzen-style formulation of intuitionistic logic is nothing but an expression of the deflection from the complete geometrical symmetry of 
Gentzen's $\bf LK$. Gentzenization of a logic also has such a sense, as far as I am concerned.

It is well-known that Le\'{s}niewski's epsilon can be interpreted by the Russellian-type definite description in classical first-order predicate 
logic with equality (see \cite{ishi}). Takano \cite{takano1} proposed a natural set-theoretic interpretation for the epsilon. 
The other motive for the above papers of mine is that I wanted to avoid such interpretations when we interpret Le\'{s}niewski's epsilon. 
I do not deny the interpretation using the Russellian-type definite description and a set-theoretic one. I was rather anxious to obtain 
another quite different interpretation of Le\'{s}niewski's epsilon having a more propositional character. We now have, in hand,  the B-translation 
of which the discovery is, in my opinion, monumental or a milestone for the future study of Le\'{s}niewski's systems. 
B-translation is very natural, if one tries to understand it in an intuitive way. However I believe that there are still other interpretations 
which are in the spirit of my original motivation. I feel this study area of the interpretation treated here has been rather neglected until now. 

Even the very simple sturcture of $\bf L_1$  gives broad possibilities of its interpretation with respect to many logics. Its study is not trivial and will 
create rich research areas. Simple subsystems of groups, for example, semigroups and quasigroups  produce fruitful results in mathematics. 
I think that $\bf L_1$ is such a mathematical system. For this, in the last section, the reader will see some problems to be researched.

In a sense, the study of interpretations of Le\'{s}niewski's epsilon in $\bf L_1$ is, as it were, the study of modal or deontic logic itself from 
another point of view, or gives an inspiration of the study. On this point, I shall eraborate it a bit more. First, we shall propose the following 
conjecture of mine:

\smallskip 

[\bf Fundamental Conjecture\rm ] (Conjecture 4 in \S 8). Le\'{s}niewski's elementary ontology $\bf L$ is embedded in some \it first-order 
\rm modal predicate logic.

\smallskip

\noindent It is well-known that $\bf L$ with axiom $\exists S (\epsilon SS)$ is embedded in monadic second-order predicate logic (see Smirnov 
\cite{smirnov1}, \cite{smirnov2} and Takano \cite{takano3}). I believe that by introuducing a modal operator, we could reduce second-order to 
first-order for the embedding. I think that this conjecture is very important and it would be the next milestone of the study in the direction of 
this paper. I here wish to focus on the second-order-like power of modal operator (logic), which would be lurked in modal logics. After I finished 
the first version of this paper, I encountered two phenomenan of the relation between modal logic and second-orderedness. One is Thomason's 
reduction of second-order logic to modal logic (see Thomason \cite{thomasonJSL, thomasonZML}). Thomason says in \cite[p. 107]{thomasonZML} as 
the following quotation:
\begin{quote}
This would seem to confirm the author's opinion [1] that "propositional modal logic (with the usual  relational semantics) must be understood as 
a rather strong fragmentt of classical second-order predicate logic".
\end{quote}
The other is Dumitru's concept in \cite{dumi} that modal imcompleteness is to be explained in terms of the incompleteness of standard second-order logic, 
since modal language is basically a second order language. In \cite{dumi}, Dumitru shows that there is a close connection between modal incompleteness 
and the incompleteness of standard second-order logic. So these papers give me the confidence of my Fundamental Conjecture. And this is, 
in a sense, a study of modal logic itself in order to understand the second-orderedness of modal logic. It is also interesting to know how Le\'{s}niewski's 
epsilon and its system is expressed and understood in topos theory through modal operators e.g. in Awodey, Kishida and Kotzsch \cite{awodey}, where 
in a certain setting, a modal operator on Heyting algebra is given in the form of a comonad. 

The model theory of this paper is based on Kobayashi and Ishimoto \cite{koishi}. In the paper, Hintikka formulas are defined in order to prove the 
completeness theorem for $\bf L_1$. Hintikka formulas appear in the end of open branches of a tableau of a formula unprovable in $\bf L_1$. 
By a Hintikka formula of a formula, say $A$, we can obtain a model which falsifies the formula $A$ for a proof of the completeness theorem. Hintikka 
formulas are standard tools for the completeness proof. In our paper, we shall use Hintikka formulas to have a model for falisification to prove the 
faithfulness of Blass translation. The novelty of this paper is that the construction of models using Hintikka formula can be easily generalized 
in order to extend our faithfulness results, especially that of deontic logics, provability logic and some normal modal logics. For automated 
reasoning, I just mention that the Hintikka formulas of $\bf L_1$ can be also used as examples of which the validity in $\bf K$ is checked 
(see e.g. Lagniez et al \cite{lima}). 

The significance of this paper will be understood as synthesis of the above mentioned faithfulness results, novelty of generalization, giving 
inspirations, some recognitions on interpretations, a refinement of the notion of chain and tail, and  presenting open problems and conjectures.

%%%%%%%%%%%%%%%%%%%%%%%%%%%%%%%

\section{Propositional ontology $\bf L_1$ and its tableau method}

Let us recall a formulation of $\bf L_1$, which was introduced in \cite{ishi}. The Hilbert-style system of it, denoted again by 
$\bf L_1$, consists of the following axiom-schemata with a formulation of classical propositional logic $\bf CP$ as its axiomatic basis:

\smallskip 

(Ax1) $\enspace$ $\epsilon ab$ $\supset$ $\epsilon aa$,

\smallskip 

(Ax2) $\enspace$ $\epsilon ab$ $\wedge$ $\epsilon bc$. $\supset$ $\epsilon ac$,

\smallskip 

(Ax3) $\enspace$ $\epsilon ab$ $\wedge$ $\epsilon bc$. $\supset$ $\epsilon ba$,

\smallskip 

\noindent where we note that every atomic formula of $\bf L_1$ is of the form $\epsilon ab$ for some name variables $a$ and $b$ and a possible intuitive interpretation of $\epsilon ab$ is `the $a$ is $b$'.

We note that (Ax1), (Ax2) and (Ax3) are theorems of Le\'{s}niewski's ontology (see S\l upecki \cite{slu}). 

The modal logic $\bf K$ is the smallest logic which contains all instances of classical tautology and all formulas of the forms $\Box (\phi \supset \psi) \supset . \Box \phi \supset \Box \psi$ being closed under modus ponens and the rule of necessitation (for $\bf K$ and basics for modal logic, see Bull and Segerberg \cite{bulseg}, Chagrov and Zakharyaschev \cite{cz}, Fitting \cite{fitting}, Hughes and Cresswell \cite{hc} and so on). 

Let us take $\vee$ (disjunction) and $\neg$ (negation) as primitive for the pure propositional logic part of $\bf L_1$. We shall employ a useful tool, i.e. \it positive $($negative$)$ parts \rm (for short, $\rm p.p.$'s ($\rm n.p.$'s)) of a formula due to Sch\"{u}tte \cite{schu1} and \cite{schutte}. The \it positive \rm and \it negative parts \rm of a formula $\phi$ of $\bf L_{1}$ are inductively defined as follows: (i) $\phi$ is a $\rm p.p.$ of $\phi$; (ii) If $\eta \vee \xi$ is a $\rm p.p.$ of $\phi$, then $\eta$ and $\xi$ are $\rm p.p.$'s of $\phi$; (iii) If $\neg \psi$ is a $\rm p.p.$ of $\phi$, then $\psi$ is a $\rm n.p.$ of $\phi$; (iv) If $\neg \psi$ is a $\rm n.p.$ of $\phi$, then $\psi$ is a $\rm p.p.$ of $\phi$.

By a notation $F [ \phi _{+}]$ ($G [ \phi _{-}]$), which also is due to \cite{schu1}, we mean that a formula $\phi$ occurs in a formula $F [ \phi _{+}]$ ($G [ \phi _{-}]$), as a $\rm p.p.$ ($\rm n.p.$) of $F [ \phi _{+}]$ ($G [ \phi _{-}]$).  Such expressions as $F [ \phi _{+}, \psi _{-}]$ and the like are understood analogously under the conditions such that for example, $F [ \phi _{+}, \psi _{-}]$, the specified formulas $\phi$ (as a $\rm p.p.$) and $\psi$ (as a $\rm n.p.$) in $F [ \phi _{+}, \psi _{-}]$ do not overlap with each other.

\begin{defi} \rm (Kobayashi-Ishimoto \cite{koishi}) 
A formula $\phi$ of $\bf L_{1}$ is said to be a \it Hintikka formula \rm of $\bf L_{1}$ if it satisfies all the following conditions: (1) $\phi$ is not of the form $F [ \psi _{+}, \psi _{-}]$; (2) If $\phi$ contains $\eta \vee \xi$ as a $\rm n.p.$ of $\phi$, then it contains $\eta$ or $\xi$ as a $\rm n.p.$ of it; (3) If $\phi$ contains $\epsilon ab$ as a $\rm n.p.$ of $\phi$, then it contains $\epsilon aa$ as a $\rm n.p.$ of it; (4) If $\phi$ contains $\epsilon ab$ and $\epsilon bc$ as $\rm n.p.$'s of $\phi$, then it contains $\epsilon ac$ as a $\rm n.p.$ of it; (5) If $\phi$ contains $\epsilon ab$ and $\epsilon bc$ as $\rm n.p.$'s of $\phi$, then it contains $\epsilon ba$ as a $\rm n.p.$ of it.
\end{defi}

%%%%%%%%%%%%%%%%%%%%%

We shall review the tableau method $\bf TL_{1}$ for $\bf L_{1}$. The tableau method $\bf TL_{1}$ for $\bf L_{1}$ is defined on the basis of the following four reduction rules:
$$\vee_{-} \enspace \frac{G [ \phi \vee \psi _{-} ]}{\enspace G [  \phi \vee \psi _{-} ] \vee \neg \phi \enspace | \enspace G [  \phi \vee \psi _{-} ] \vee \neg \psi \enspace }$$
\smallskip
$$\epsilon_{1} \enspace \frac{G [ \epsilon ab _{-} ]}{\enspace G [ \epsilon ab _{-} ] \vee \neg \epsilon aa \enspace }$$
\smallskip
$$\epsilon_{2} \enspace \frac{G [ \epsilon ab _{-}, \epsilon bc _{-} ]}{\enspace G [ \epsilon ab_{-}, \epsilon bc _{-} ] \vee \neg \epsilon ac \enspace }$$
\smallskip
$$\epsilon_{3} \enspace \frac{G [ \epsilon ab _{-}, \epsilon bc _{-} ]}{\enspace G [ \epsilon ab_{-}, \epsilon bc _{-} ] \vee \neg \epsilon ba \enspace }$$

\smallskip

\noindent where for $\bf TL_{1}$, we take the same logical symbols as primitive as for $\bf L_{1}$. (This tableau method $\bf TL_{1}$ is due to \cite{koishi}.)

By reducing a formula by these reduction rules, we obtain a \it tableau \rm of the formula. A branch of a tableau is \it closed \rm if it ends with a formula of the form $F [ \phi _{+}, \phi _{-}]$, which is called an \it axiom \rm of $\bf TL_{1}$. A branch of a tableau is \it open \rm if it is not closed. A tableau is said to be \it closed \rm if every branch of it is closed, otherwise we call it \it open\rm . A tableau is said to be \it finite\rm , if it has a finite number of branches and evey branch of it is obtained by a finite numbers of applications of reduction rules. A formula of $\bf TL_{1}$ is \it provable in $\bf TL_{1}$ \rm if there exists a closed tableau of it.

\begin{theo} \rm (\cite{koishi} (The Fundamental Theorem for $\bf TL_1$) \it Given a formula \rm (\it of $\bf L_{1}$ \rm ($\bf TL_{1}$)), \it by reducing it with reduction rules we obtain a finite tableau, each branch of which ends either with an axiom of $\bf TL_{1}$ or with a Hintikka formula, where every branch of it is extended by a reduction rule only if the formula to be reduced, say $\phi$, is not an axiom of $\bf TL_{1}$ and the reduction gives rise to a formula not occurring in $\phi$ as a negative part of $\phi$.
\end{theo}

A tableau is said to be \it normal\rm , if it is constructed in compliance with the condition of Theorem 2.2.

If a normal tableau has a Hintikka formula $\phi$ at the end of some branch of it (in other words, if the formula to be reduced is not provable in $\bf TL_{1}$), 
then it is obvious from its normality that $\phi$ cannot be reduced further. Note that a normal tableau of a formula is not always the shortest one of the 
possible tableaux of it.

\begin{defi} \rm 
Let $\phi$ be a formula of $\bf L_{1}$.
For any formula $\phi$ of $\bf L_{1}$, we call $\psi$ to be \it a Hintikka formula of $\phi$\rm , if one of open branches of a normal tableau of $\phi$ ends 
with $\psi$.
\end{defi}

\begin{theo} \rm (\cite{inoue2} and \cite{koishi}) \it 
For any formula $\phi$ of $\bf L_{1}$, we have
$$\vdash _{\bf TL_{1}} \phi \Leftrightarrow \mathrm{ } \vdash _{\bf L_{1}} \phi.$$
\end{theo}

A $\rm p.p.$ is said to be \it minimal \rm if it is not of the form $\neg \phi$ or $\phi \vee \psi$. A $\rm n.p.$ is said to be \it mininal \rm 
if it is not of the form $\neg \phi$. (See \cite[$\rm p. 12$]{schutte}).

We need the following theorem for the proof of our main theorem.

\begin{theo} \rm (\cite[$\rm p. 12$]{schutte}) \it 
If, under a sentential valuation $v$, every minimal $\it p.p.$ of a formula $\phi$ takes the value $0$ $($falsity$)$ and every minimal $\it n.p.$ of $\phi$ takes the value $1$ $($truth$)$, then every $\it p.p.$ of $\phi$ takes the value $0$ and every $\it n.p.$ of $\phi$ takes the
value $1$. $($In particular, $v(\phi) = 0$ holds, since $\phi$ is a $\it p.p.$ of $\phi$.$)$
\end{theo}

%%%%%%%%%%%%%%%%%%%%%%%%%%%%%%%%%%%%%%

\section{Chain equivalence relation, one more preliminary}

As one more preliminary, we introduce the notion of chain equivalence classes and tails of a Hintikka formula, which are important for our proofs.

\begin{defi} \rm 
For any formula $\phi$ of $\bf L_1$, \it the  name vaiable set $NV_{\phi}$ \rm of $\phi$ is 
defined as follows:
$$NV_{\phi} = \{ a : a \mbox{ is a name variable of } \phi \}.$$
\end{defi}

We note that for any formula $\phi$ of $\bf L_1$, for any Hintikka formula $\psi$ of $\phi$, $NV_{\phi} = NV_{\psi}$ holds.

\begin{defi} \rm 
For any Hintikka formula $\phi$ of $\bf L_1$, \it the  chain name variable set $CN_{\phi}$ \rm of $\phi$ is 
defined as follows:
$$CN_{\phi} = \{ a \in NV_{\phi} : \exists b \in NV_{\phi} (\phi = G [ \epsilon ab _{-}, \epsilon ba _{-}]) \}.$$
\end{defi}

It is possible that $CN_{\phi}$ is empty for some Hintikka formula $\phi$, e.g. $\phi = \epsilon aa$.

\begin{defi} \rm 
For any Hintikka formula $\phi$ of $\bf L_1$, we define a relation on the chain name vaiable set $CN_{\phi}$ \rm of $\phi$ is 
defined as follows:
$$\sim_{\phi} = \{ (a, b) \in CN_{\phi} \times CN_{\phi} : \phi = G [ \epsilon ab _{-}, \epsilon ba _{-}]  \}.$$
\end{defi}

\begin{prop} \it
For any Hintikka formula $\phi$ of $\bf L_1$, the relation $\sim_{\phi}$ on $CN_{\phi}$ is an equivalence relation.
\end{prop}
\begin{proof}
Let $\phi$ be a Hintikka formula of $\bf L_1$. For any $a \in CN_{\phi}$, $G[ \epsilon aa _{-}] = \phi$ holds by Definition 2.1.(3) and Definition 3.3. 
So we have the reflexivity $a \sim_{\phi} a$. For any $a, b \in CN_{\phi}$, the symmetry $(a \sim_{\phi} b  \Rightarrow \mathrm{ } b \sim_{\phi} a$) 
trivially holds by Definition 3.3 and the definition of negative parts.  For any $a, b, c \in CN_{\phi}$, we have 
$$a \sim_{\phi} b, b \sim_{\phi} c
\Rightarrow \phi =  G_{1} [ \epsilon ab _{-}, \epsilon ba _{-}, \epsilon  bc_{-}, \epsilon cb _{-}].$$
Apply Definition 2.1.(4) and Definition 3.3 to pairs $(\epsilon ab, \epsilon bc)$ and $(\epsilon cb, \epsilon ba).$
Then,  $\phi =  G_{2} [ \epsilon ac _{-}, \epsilon ca _{-}].$ This means $a \sim_{\phi} c.$
So we obtain the transitivity. Hence the relation is an equivalence relation. 
\end{proof}

\begin{defi} \rm 
Let $\phi$ be a Hintikka formula of $\bf L_1$. For any $a \in CN_{\phi}$, the equivalence class of $a$ with respect to $\sim_{\phi}$, $$[a]_{\sim_{\phi}} = \{b \in CN_{\phi} : a \sim_{\phi} b\}$$
is the set of all elements of $CN_{\phi}$ equivalent to $a$, which we call \it the chain equivalence class of $a$ \rm(\it with respect to $\sim_{\phi}$\rm).
\end{defi}

\begin{defi} \rm 
For any Hintikka formula $\phi$ of $\bf L_1$, we define "$CN_{\phi}$ modulo $\sim_{\phi}$'' to be the set $$ChainQ(\phi) = CN_{\phi}  \\/ \sim_{\phi} = \{ [a]_{\sim_{\phi}} : a \in CN_{\phi} \}$$ 
of all chain equivalence classes of $CN_{\phi}$. We call $ChainQ(\phi)$ \it the chain quotient set of $\phi$.\rm We call an element of $ChainQ(\phi)$ a \it chain of $\phi$.
\end{defi}

\begin{defi} \rm (\cite{koishi}) 
For any Hintikka formula $\phi$ of $\bf L_1$, a \it Kobayashi-Ishimoto-chain \rm of $\phi$ is a nonempty (finite) set $C$ of name variables, say $C = \{a_1, a_2, \dots , a_n\}$ ($n \geq 1$) such that: (1) Every pair $a_i$ and $a_j$ ($1 \leq i \leq n$, $1 \leq j \leq n$) belonging to $C$ are connected by a relation defined as $\epsilon a_ia_j$ and $\epsilon a_ja_i$ both of which are $\rm n.p.$'s of $\phi$;(2) The set is maximal with respect to the property in (1).
\end{defi}

\begin{defi} \rm 
For any Hintikka formula $\phi$ of $\bf L_1$, we define \it Kobayashi-Ishimoto-chain set 
$ChainKI(\phi)$ \rm of $\phi$ as follows:
$$ChainKI(\phi) = \{ C : C \mbox{ is a Kobayashi-Ishimoto-chain of } \phi \}.$$
\end{defi}

\begin{prop}  
For any Hintikka formula $\phi$ of $\bf L_1$, we have $ChainQ(\phi) = ChainKI(\phi).$
\end{prop}
\begin{proof} Let $\phi$ be a Hintikka formula of $\bf L_1$. We first prove $ChainQ(\phi) \subseteq ChainKI(\phi)$.
Let $C$ be an element of $ChainQ(\phi)$. Take a name variable $a$ such that $C = [a]_{\sim_{\phi}}$ holds. 
For any $x, y \in C$, $\phi =  G [ \epsilon xy _{-}, \epsilon yx _{-}]$ holds, since 
$\phi =  G_{1} [ \epsilon xa _{-}, \epsilon ax _{-}, \epsilon  ya_{-}, \epsilon ay _{-}].$ Apply Definition 2.1.(4) and Definition 3.3 to pairs
 $(\epsilon xa, \epsilon ay)$ and $(\epsilon ya, \epsilon ax)$.
Then, $\phi =  G [ \epsilon xy _{-}, \epsilon yx _{-}].$
This satisfies Definition 3.7.(1). By Definition 3.6, $C \in ChainQ(\phi)$ is an element of a partition of $CN_{\phi}$. 
So $C$ satisfies Definition 3.7.(2). Hence $C \in ChainKI(\phi)$. 

We shall prove $ChainKI(\phi) \subseteq ChainQ(\phi)$. Let $C$ be an element of $ChainKI(\phi)$. By Definition 3.6, we 
have $C \subseteq CN_{\phi}.$ Take an element $a$ of $C$, arbitrarily. Then, by Defintion 3.7.(1), 
$\phi =  G [ \epsilon ax _{-}, \epsilon xa _{-}]$ holds for any $x \in C$. We obtain $C \subseteq [a]_{\sim_{\phi}}$. Let $y \in [a]_{\sim_{\phi}}$. 
The maximality of $C$ follows $ y \in C$. So, $C = [a]_{\sim_{\phi}}$. Thus $C \in ChainQ(\phi)$. 
\end{proof}

Thus, a Kobayashi-Ishimoto-chain is nothing but our chain.  Our Definition 3.6 clarifies the notion of Kobayashi-Ishimoto-chains.

\begin{defi} \rm (\cite{koishi}) 
For any Hintikka formula $\phi$ of $\bf L_1$ and any chain $C$ of $\phi$, a \it tail \rm of $C$ is a name variable $b$ such that $\epsilon ab$ 
is a $\rm n.p.$ of $\phi$ with $a \in C$ and $b \notin C$. By $TN_{\phi}$, we denote the set of all tails of a Hintikka formula $\phi$.
\end{defi}

We may have a chain of a given Hintikka formula $\phi$ without tails, e.g. $\phi = \neg \epsilon aa$.

\begin{prop} \it 
For any Hintikka formula $\phi$ of $\bf L_1$ and any chain $C$ of $\phi$, no tail of $C$ belongs to any other chains of $\phi$.
\end{prop}
\begin{proof} Let $\phi$ be a Hintikka formula of $\bf L_1$. Let $C_1$ and $C_2$ be distinct chains of $\phi$ and $b$ a tail of $C_1$. 
We may suppose that they are nonempty. 
Then there is a name variable $a$ of $C_1$ such that $\epsilon ab$ is a $\rm n.p.$ of $\phi$. 
Suppose that $b$ is a member of $C_2$. Since $b$ is an element of $C_2$, by Definition 3.2, there is some name variable, say $c$ such that $\epsilon bc$ and
 $\epsilon cb$ are n.p.'s of $\phi$.  Since  $\epsilon ab$ is a $\rm n.p.$ of $\phi$,  by Definition 2.1.(5), $\phi$ contains $\epsilon ba$ as its $\rm n.p.$. 
 In other words, $b$ is a member of $C_1$, which contradicts the definition of tails. 
\end{proof}

\begin{defi} \rm 
Let $\phi$ be a Hintikka formula of $\bf L_1$. By $Rest_{\phi}$, we denote $NV_{\phi} - (CN_{\phi} \cup TN_{\phi})$.
\end{defi}

We can analyse the set of name variable occurring in the set of all minimal (atomic) p.p.'s and n.p.'s of a Hintikka formula of $\bf L_1$.
And it plays an important role for the proof of our main theorem.

\begin{theo} \rm (Characterization Theorem for name variables of a Hintikka formula) \it 
For any Hintikka formula $\phi$ of $\bf L_1$ and any name variable $a \in Rest_{\phi}$, $a$ occurs in some minimal p.p. of $\phi$. 
Further, 
$$NV_{\phi} = CN_{\phi} \amalg TN_{\phi} \amalg Rest_{\phi} \enspace \enspace (*)$$
holds, where $\amalg$ means a disjoint union. 
\end{theo}
\begin{proof} Let $\phi$ be a Hintikka formula of $\bf L_1$. From Definition 3.10 and Proposition 3.11, we have 
$CN_{\phi} \cup TN_{\phi} = CN_{\phi} \amalg TN_{\phi}.$ From Definition 3.12, this leads to $(*)$.
\end{proof}

\begin{lem} \rm (Tail Lemma) \it
Let $\phi$ be a Hintikka formula of $\bf L_1$. For any tail $a$ of a Chain of $\phi$ and any 
$b \in NV_{\phi}$, $\epsilon ab$ does not occur as a n.p. of $\phi$.
\end{lem}
\begin{proof} Let $\phi$ be a Hintikka formula of $\bf L_1$. Let $a$ be a tail of a chain of 
$\phi$ and $b \in NV_{\phi}$. Suppose $\phi = G[ \epsilon ab_{-}]$. Since $\phi$ is a 
Hintikka formula of $\bf L_1$, we have $\phi = G_1[ \epsilon ab_{-}, \epsilon aa_{-}]$. This 
means $a \in [a]_{\sim_{\phi}}$. That is, $a$ is an element of the chain $[a]_{\sim_{\phi}}$. 
This contradicts Proposition 3.11.
\end{proof}

%%%%%%%%%%%%%%%%%%%%%%%%%%%%%%%%%%%%%%%%%%%%%%%%%%%%

\section{Another proof of the faithfulness of $B$-translation using Hintikka formula}

In this section, we shall prove the faithfulness of $B$-translation by means of Hintikka formula with respect to $\bf K$, that is:

\begin{theo} \rm (Blass \cite{blass}) \it 
For any formula $\phi$ of $\bf L_1$, 

\noindent we have
$$(\clubsuit) \enspace \mbox{\it\rm } \vdash_{\bf K} B(\phi)  \Rightarrow \enspace \vdash _{\bf L_1} \phi.$$
That is, $B$-translation is faithful with respect to $\bf K$.
\end{theo}
\begin{proof} Let $\phi$ be a formula of $\bf L_{1}$. For the proof, it is sufficient to show  that (not $\vdash _{\bf TL_{1}} \phi$) $\Rightarrow$ 
(not $\vdash _{\bf K} B(\phi)$). 
Then by Theorem 2.4 we obtain the desired meta-implication in $(\clubsuit)$. If $\phi$ is a theorem of $\bf TL_{1}$, then we trivially have 
the meta-implication. Suppose that $\phi$ is not a theorem of $\bf TL_{1}$. 
Then, there exists an open normal tableau such that it has a branch ended with a Hintikka formula, say $\psi$. We immediately see that $\phi$ 
is a $\rm p.p.$ of $\psi$ (observe the reduction rules for $\bf TL_{1}$). 
(We can eventually prove it by induction on derivation.) That is, $\psi$ is a Hintikka formula of $\phi$. 
We may choose such a Hintikka formula arbitrarily as a Hintikka formula of $\phi$ for our model construction below.

Let $v$ be a sentential valuation such that by $v$, every atomic $\rm p.p.$ ($\rm n.p.$) of $\psi$ is assigned falsity 0 (truth 1) 
(the rest of the assignment is at a person's disposal), where we do not need to look 
into the structure of atomic formulas such as $\epsilon ab$. Since $\psi$ is a Hintikka formula, the valuation $v$ makes every minimal 
$\rm p.p.$ ($\rm n.p.$) false (true). Hence by Theorem 2.5, if we have 
$v(\psi) = 0$, then $v(\phi) = 0$ holds, since $\phi$ is a $\rm p.p.$ of $\psi$. 

Since $B(\cdot)$ commutes disjunction and negation, $B(\phi)$ in the setting of $\bf K$ is falsified by the (adapted) valuation $v$ such that 
we have $v(\psi) = 0$ in that of $\bf L_{1}$. In other words, in $B(\phi)$ 
we can regard formulas of the form $B(\epsilon  ab)$ as atomic formulas, when we assign truth-values to them in order to falsify $B(\phi)$.
 
We shall below construct such a valuation, namely a model in which $B(\phi)$ is false. (If we have a model falsifying $B(\phi)$, then by the 
completeness theorem for $\bf K$, $B(\phi)$ is not a theorem of $\bf K$.) 
Given a Hintikka formula $\psi$, we have a finite numbers of its chains and their associated tails, observing its atomic negative and positive parts. 
Say $ChainQ(\psi) = \{C_1, \dots , C_n \}$ ($n \geq 0$). If $n = 0$ 
(that is, $\psi$ contains no atomic negative parts), then we take a Kripke model $\cal M_{\rm 0}$ =$<G_0 , R_0 , V_0>$ such that 

$G_0 = \{*, g\} \cup G_{\omega}$ ($* \neq g$), $R_0 = \{(*, g) \cup \{(\eta, \eta) : \eta \in  G_{\omega} \}$.

\noindent $V_0(p_a, x) = 0$ for any name variable $a$ and any $x \in G_0$, where 

$G_{\omega} = \{\eta_{j}: j \in \omega\}$, $Card(G_{\omega}) = \aleph_0$, $\{*, g \} \cap G_{\omega} = \emptyset$.

\noindent ($V_0(p, x)$ stands for the truth-value of a propositional variable $p$ in a world $x$ (by $V_0$) in $\cal M_{\rm 0}$). 
Then we easily see that $B(\phi)$ is false in $*$ in $\cal M_{\rm 0}$, since $B(\epsilon ab)$, that is, 
$$p_a \wedge \Box (p_a \supset p_b) \wedge. p_b \supset \Box (p_b \supset p_a)$$ is trivially false in $*$ in $\cal M_{\rm 0}$ regardless of 
the modality.
So suppose $n \geq 1$. 
Now we shall construct a Kripke model $\cal M$ =$<G , R , V>$ (below we shall write $g \models \phi$ for 
$\models^{\cal M}_g \phi$ (i.e., $\phi$ is true in a world $g$ in $\cal M$) as follows.

\smallskip 

\noindent (M1) $G = \{*, g_1, \dots , g_n\} \cup G_{\omega}$, 

\smallskip 

\noindent where $G_{\omega} = \{\eta_{j}: j \in \omega\}$, $Card(G_{\omega}) = \aleph_0$, 
$\{*, g_1, \dots , g_n\} \cap G_{\omega} = \emptyset$ and $*, g_1, \dots , g_n$ are distinct.

\noindent (M2) $R = \{(*, g_i ) : 1 \leq i \leq n\} \cup \{(\eta, \eta) : \eta \in  G_{\omega} \}$.

\noindent (M3) 

(i) For any name variable $a \in NV_{\psi}$,  
$$V(p_a, *) = \left\{\begin{array}{ll} 1 & \mbox{if $a \in CN_{\psi}$,}\\ 0 & \mbox{otherwise.}
\end{array} \right.$$

(ii) Take one name variable $c_i$ from each $C_i$ ($1 \leq i \leq n$) as its representative, 
that is, say $C_i = [c_i]_{\sim_{\psi}}$.
 Assume that if the $p_{c_i}$ is assigned a truth-value in a world, then every propositional variable $p$ 
 corresponding to the other name variable belonged to $C_i$ should be assigned the same truth-value in the same world. 
 Under this assumption, we take the following assignment:
$$\left(\begin{array}{cccc}
V(p_{c_1}, g_1) & V(p_{c_1}, g_2) & \ldots & V(p_{c_1}, g_n) \\
V(p_{c_2}, g_1) & V(p_{c_2}, g_2) &\ldots & V(p_{c_2}, g_n) \\
\vdots & \vdots & \ddots & \vdots \\
V(p_{c_n}, g_1) & V(p_{c_n}, g_2) & \ldots & V(p_{c_n}, g_n) 
\end{array} \right)  = U$$
where $U$ is $n \times n$ unit matrix in linear algebra. (Here we remark that for this construction of Kripke model, we do not need the notion of 
connectedness used in \cite{inoue3}.) We remark that the order of $C_1, \dots , C_n$ does not matter.

(iii) If a name variable $a$ is a tail of some chain, we take the following assignment. First fix the assingment of (M3).(ii). 
Take all chains such that $a$ is a tail of them, say $\tilde{C}_1, \dots , \tilde{C}_M$ $(0 \leq M \leq n)$ (if there is a tail, 
then $M \geq 1$). Choose  $d_i \in \tilde{C_i}$ for any $1 \leq i \leq M$, arbitrarily. Take all indices $\xi_1, \dots , \xi_M$ such 
that $V(p_{d_i}, g_{\xi_i}) = 1$ for any $1 \leq i \leq M$ (note that $d_i$ may be eventually equal to $c_{\xi_i}$ in 
the notation of (M3).(ii)). Such $\xi_1, \dots , \xi_M$ exist because of (M3).(ii), Proposition 3.4 and Definition 3.6. 
Then take, for any $1 \leq i \leq n$,
$$V(p_a, g_i) = \left\{\begin{array}{ll} 1 & \mbox{if $i \in \{\xi_1, \dots , \xi_M\}$,}\\ 0 & \mbox{otherwise.}
\end{array} \right.$$

(iv)  For any $p$ of the rest of propositional variables and any $1 \leq i \leq n$, $V(p, g_i) = 0$.

(v) For any propositional variable $p$ and any $\eta \in G_{\omega}$, $V(p, \eta) = 1$.

Now, let us verify that the model just constructed above actually falsifies $B(\psi)$ in the world $*$, which makes 
$B(\psi)$ invalid. So $\not\models_{\bf K} B(\phi)$ holds from Theorem 2.5. Thus, $\not\vdash_{\bf K} B(\phi)$ by the completeness theorem 
for $\bf K$. Let $D_1, \dots , D_s$ be all distinct atomic $\rm p.p.$'s of $\psi$ and $E_1, \dots , E_t$ all distinct atomic 
$\rm n.p.$'s of $\psi$ ($s + t \geq 1$, $t \geq 1$) because of $n \geq 1$. 
In order to apply Theorem 2.5 to our case, we have to show that

\smallskip

$(\dagger)$ $\enspace$ $* \not\models B(D_k)$ for any $1 \leq k \leq s$ ($s$ may possibly be 0.),

\smallskip

$(\dagger \dagger)$ $\enspace$ $* \models B(E_k)$ for any $1 \leq k \leq t$.

\smallskip

First we remark that we do not need to consider the worlds of $G_{\omega}$ at all, since $*$ does not relate any 
element of $G_{\omega}$. We put the $G_{\omega}$ in order to keep a generality of models.

For verifying $(\dagger)$ and $(\dagger \dagger)$, we need to classify all the name variables (in minimal parts of $\psi$) occurring in $\psi$, 
thus in $\phi$ as follows.

(NV1) All distinct name variables occurring in chains, say $x_1, \dots , x_p$ ($p \geq 1$).
Note $CN_{\psi} =  \{x_1, \dots , x_p\}$.

(NV2) All distinct name variables which are tails of some chain, say $y_1, \dots , y_q$ ($q \geq 0$). 
So, $TN_{\phi} = \{y_1, \dots , y_q\}$.

(NV3) All distinct name variables which are not tails, occurring in some atomic positive parts of $B$, but not in any 
of chains, say $z_1, \dots , z_r$ ($r \geq 0$), that is, $Rest_{\phi}= \{z_1, \dots , z_r\}$.

Note that $x_1, \dots , x_p, y_1, \dots , y_q, z_1, \dots , z_r$ are mutually distinct variables because of Theorem 3.13.

We shall first verify $(\dagger)$. There are the following cases (Case 1)--(Case 5) for $D_k$ ($1 \leq k \leq s$).  
Suppose $s \geq 1$. If $s = 0$, we ignore the verification for the cases.. 

(Case 1): The case of $D_k = \epsilon x_ix_j$ or $\epsilon x_jx_i$ for any $i$ and $j$ with $i < j$, such that $x_i$ and $x_j$ are not 
in the same chain. (Because of Definitions 2.1.(1) and 3.5, it is not possible that $x_i$ and $x_j$ belong to the same chain, 
if $\epsilon x_ix_j$ is a $\rm p.p.$ of $\psi$. We do not need to consider the case of $i = j$ because of Definition 2.1.(1).)  
In this case, by (M3).(ii), there are exactly two indices $\alpha$ and $\beta$ ($\alpha \neq \beta$) ($1 \leq \alpha \leq n$, $1 \leq \beta \leq n$) 
such that 

(a) $x_i \in [c_{\alpha}]$, $g_\alpha \models p_{x_i}$, $g_\alpha \not\models p_{x_j}$, 

(b) $x_j \in [c_{\beta}]$, $g_\beta \models p_{x_j}$, $g_\beta \not\models p_{x_i}$ 

\noindent hold. So we have $g_\alpha \not\models p_{x_i} \supset p_{x_j}$ and $g_\beta \not\models p_{x_j} \supset p_{x_i}$. Since 
$$* \models \Box (p_{x_i} \supset p_{x_j}) \Leftrightarrow .  
g_1 \models p_{x_i} \supset p_{x_j} \wedge  g_2 \models p_{x_i} \supset p_{x_j} \wedge \cdots \wedge g_n \models p_{x_i} \supset p_{x_j}$$ 
holds, we obtain $* \not\models \Box (p_{x_i} \supset p_{x_j})$, when $D_k = \epsilon x_ix_j$. Similarly, we get 
$* \not\models \Box (p_{x_j} \supset p_{x_i})$, when $D_k = \epsilon x_jx_i$. Thus we have $* \not\models B(D_k)$ in both cases.

(Case 2): The case of $D_k = \epsilon x_iy_j$ for any $i$ and $j$ such that $y_j$ is not a tail of the chain to which $x_i$ belongs. Then, 
by (M3).(ii) and (M3).(iii), there is exactly one index $1 \leq \alpha \leq n$ 
such that $V(p_{x_i}, g_\alpha) = 1$ and $V(p_{y_j}, g_\alpha) = 0$. From this, we have $* \not\models \Box (p_{x_i} \supset p_{y_j})$. 
So we get $* \not\models B(D_k)$. 

(Case 3): The case of $D_k = \epsilon x_iz_j$ for any $i$ and $j$. In this case we have 
$* \models p_{x_i}$ and $* \models p_{z_j} \supset . \Box (p_{z_j} \supset p_{x_i})$.
But  $* \not\models \Box (p_{x_i} \supset p_{z_j})$ holds as (Case 1). 

(Case 4): The case of $D_k = \epsilon y_ib$ for any $i$ where $b$ is arbitrary.  From (M3).(i) and Theorem 3.13, $* \not\models p_{y_i}$ holds.
This makes our case verified.

(Case 5): The case of $D_k = \epsilon z_ib$ for any $i$ where $b$ is arbitrary. Then, by (M3).(i), $* \not\models p_{z_i}$, thus $* \not\models B(D_k)$.

Next, we shall verify $(\dagger \dagger)$. In this case, there are only the following two cases for $E_k$ ($1 \leq k \leq t$).

(Case 6): The case of $E_k = \epsilon x_ix_j$ for any $i$ and $j$ such that $x_i$ and $x_j$ belong to the same chain (because $\psi$ is a Hintikka formula). 
By (M3).(i), $* \models p_{x_i}$. By (M3).(ii), we have $* \models \Box (p_{x_i} \supset p_{x_j})$ and $* \models p_{x_j} \supset \Box (p_{x_j} \supset p_{x_i})$, 
since $g_\alpha \models p_{x_i} \equiv p_{x_j}$ holds for any $(1 \leq \alpha \leq n)$. Thus, $* \models B(E_k)$.

(Case 7): The case of $E_k = \epsilon x_iy_j$ for any $i$ and $j$ such that $y_j$ is a tail of the chain to which $x_i$ belongs (because $\psi$ is a Hintikka formula). 
By (M3).(i), $* \models p_{x_i}$. By (M3).(ii) and (M3).(iii), we obtain $* \models \Box (p_{x_i} \supset p_{y_i})$. 
By (M3).(i), we have $* \models p_{y_j} \supset \Box (p_{y_j} \supset p_{x_i})$, since 
$* \models p_{y_{i}} \mbox{   and   } * \not\models \Box (p_{y_j} \supset p_{x_i})$ 
\noindent holds as (Case 6). Hence, $* \models B(E_k)$.

We can now conclude $* \not\models B(\psi)$. This completes the whole proof for $(\clubsuit)$. 
\end{proof}

\begin{coro} \rm (Blass \cite{blass}) \it 
$B$-translation is an embedding of $\bf L_1$ in $\bf K$.
\end{coro}
\begin{proof}
From Theorem 4.1.
\end{proof}

%%%%%%%%%%%%%%%%%%%%%

\section{The faithfulness for von Wright-type deontic logics}

As von Wright-type deontic logic, we shall deal with ten Smiley-Hanson systems of monadic deontic logic after \AA qvist \cite{aq}, 
that is, $\bf OK$, $\bf OM$, $\bf OS4$, $\bf OB$, $\bf OS5$, $\bf OK^{+}$, $\bf OM^{+}$, $\bf OS4^{+}$, $\bf OB^{+}$, $\bf OS5^{+}$. 

As primitive logical connectives, we take $\top$ (verum), $\bot$ (falsum), $\neg$ (nagation), $O$ (obligation), $P$ (permission), $\wedge$ (conjunction), 
$\vee$ (disjunction), $\rightarrow$ (implication), $\leftrightarrow$ (material equivalence). (We may think of $O$ and $P$ as $\Box$ and $\Diamond$, 
respectively.) 

The well-formed formulas of each system are defined as usual as those of propositional modal logics. 

The two rules of inferences, \it modus ponens \rm and \it $O$-necessitation \rm ($\vdash A$ implies $\vdash OA$) are common to all the ten Smiley-Hanson 
systems of monadic deontic logic.

We need the following axiom schemata for the system.

\smallskip

(A0) All classical propositional tautologies

(A1) \enspace \enspace $PA \leftrightarrow \neg O \neg A$

(A2) \enspace \enspace $O (A \rightarrow  B) \rightarrow (OA \rightarrow OB)$

(A3) \enspace \enspace $OA \rightarrow PA$

(A4) \enspace \enspace $OA \rightarrow OOA$

(A5) \enspace \enspace $POA \rightarrow OA$

(A6) \enspace \enspace $O(OA \rightarrow A)$

(A7) \enspace \enspace $O(POA \rightarrow A)$

\smallskip

First five systems are defined as follows.

\smallskip

$\bf OK$ \enspace = \enspace \enspace A0--A2

$\bf OM$ \enspace = \enspace \enspace A0--A2, A6

$\bf OS4$ \enspace = \enspace \enspace A0--A2, A4, A6

$\bf OB$ \enspace = \enspace \enspace A0--A2, A6, A7

$\bf OS5$ \enspace = \enspace \enspace A0--A2, A4, A5

\smallskip

\noindent Note that A6 and A7 are derivable in $\bf OS5$. 

Let $X$ be any of these five systems. Then we define:

\smallskip

$X^{+}$ \enspace = \enspace $X$, A3.

\smallskip

We shall recall the definition of accessibility relations as follows.

(AR1) \enspace \enspace $R$ is serial in $W$ \enspace \enspace \enspace \enspace \enspace \enspace $\forall x \exists (xRy)$

(AR2) \enspace \enspace $R$ is transitive in $W$ \enspace \enspace \enspace $\forall x \forall y \forall z (xRy \wedge yRz . \supset xRz)$

(AR3) \enspace \enspace $R$ is Euclidean in $W$ \enspace \enspace \enspace $\forall x \forall y \forall z (xRy \wedge xRz . \supset yRz)$

(AR4) \enspace \enspace $R$ is almost reflexive in $W$ \enspace \enspace \enspace $\forall x \forall y \forall z (xRy \supset yRy)$

(AR5) \enspace \enspace $R$ is almost symmetric in $W$ \enspace \enspace \enspace $\forall x \forall y \forall z (xRy \supset . yRz \supset zRy)$, 

\noindent where $W$ is the set of possible worlds. With these relations we characterize the systems as follows:

The class of $\bf OK$-models has no condition $R$ being imposed.

The class of $\bf OM$-models has almost reflexive $R$.

The class of $\bf OS4$-models has transitive and almost reflexive $R$.

The class of $\bf OB$-models has almost symmetric and almost reflexive $R$.

The class of $\bf OS5$-models has Euclidean and transitive $R$.

The class of $\bf OK^{+}$-models has serial $R$.

The class of $\bf OM^{+}$-models has serial and almost reflexive $R$.

The class of $\bf OS4^{+}$-models has serial, transitive and almost reflexive $R$.

The class of $\bf OB^{+}$-models has serial, almost symmetric and almost reflexive $R$.

The class of $\bf OS5^{+}$-models has serial, Euclidean and transitive $R$.

\begin{theo} \rm (See \cite{aq}) \it 
The soundness and completeness theorems hold for any ten Smiley-Hanson systems.
\end{theo}

We shall take trivial adaptation $B^O$ of $B$-tranlation for the Smiley-Hanson systems.

\smallskip

(O.i) \enspace $B^O(\phi \vee \psi)$ = $B^O(\phi) \vee B^O(\psi)$, 

(O.ii) \enspace $B^O(\neg \phi)$ = $\neg B^O(\phi)$, 

(O.iii) \enspace $B^O(\epsilon ab)$ = $p_a \wedge O(p_a \supset p_b) \wedge. p_b \supset O(p_b \supset p_a)$,

\smallskip 

\noindent where $p_a$ and $p_b$ are propositional variables corresponding to the name variables $a$ and $b$, respectively.

We may regard $\bf OK$ as $\bf K$. The other systems are stronger than $\bf OK$. Hence the following theorem is easily proved 
model theoretically. 

\begin{theo} 
For any of ten Smiley-Hanson systems, say $\bf SH$, and any formula $\phi$ of $\bf L_{1}$, we have 
$\vdash _{\bf TL_{1}} \phi \Rightarrow \mathrm{ } \vdash _{\bf SH} B^O(\phi)..$
\end{theo}

With the slight modification of the accessibility relations as above, we can obtain following theorem for 
the faithfullness of the adapted $B^O$-translation. 

\begin{theo} 
For any of ten Smiley-Hanson systems, say $\bf SH$, and any formula $\phi$ of $\bf L_{1}$, we have 
$\vdash _{\bf SH} B^O(\phi) \Rightarrow \mathrm{ } \vdash _{\bf L_{1}} \phi.$
\end{theo}
\begin{proof} Let $\bf SH$ be one of ten Smiley-Hanson systems. Let $\phi$ be a formula of $\bf L_{1}$. 
If the number of chain $n \geq 1$ holds, we take, instead of the original $G$ and $R$ in the model construction of Theorem 4.1, 

\smallskip 

\noindent ($\mathrm{M^{O}}$1) $G^O = \{*, g_1, \dots , g_n\} \cup G_{\omega}$. 

\smallskip 

\noindent where $G_{\omega} = \{\eta_{j}: j \in \omega\}$, $Card(G_{\omega}) = \aleph_0$, 
$\{*, g_1, \dots , g_n\} \cap G_{\omega} = \emptyset$ and $*, g_1, \dots , g_n$ are distinct.

\noindent ($\mathrm{M^{O}}$2) 

$$R^O = \left\{\begin{array}{lll} BR \cup \{(g_i, g_j): i \neq j, 1 \leq i \leq n, 1 \leq j \leq n \} & 
\mbox{if $\bf SH$ is } 
\\ $$ & \mbox{\enspace \enspace \enspace $\bf OS5$ or $\bf OS5^{+}$}
\\ BR \cup \{(g_i, g_i): 1 \leq i \leq n \} & \mbox{otherwise.}
\end{array} \right.$$
\noindent where $BR = \{(*, g_i ) : 1 \leq i \leq n\} \cup \{(\eta, \eta) : \eta \in  G_{\omega} \}$.

If $n = 0$ holds, we take $G^O_0 = \{*, g\} \cup G_{\omega}$ as  ($\mathrm{M^{O}}$1). 
$$R^O_0 = \left\{\begin{array}{lll} \{(*, g) \} \cup \{(\eta, \eta) : \eta \in G_{\omega} \}. & 
\mbox{if $\bf SH$ is } 
\\ $$ & \mbox{\enspace \enspace \enspace $\bf OS5$ or $\bf OS5^{+}$}
\\ \{(*, g), (g, *), (g, g) \} \cup \{(\eta, \eta) : \eta \in G_{\omega} \}.  & \mbox{otherwise.}
\end{array} \right.$$

We shall take the same assignment as that for Theorem 4.1. The above modification of the accessibility relations 
does not effect the truth of $* \not\models B^O(\phi)$, since the additional parts of relations do not related to $*$. 
Hence we prove this theorem with the same procedures of Theorem 4.1. We remark that we can take the following relation 
$$R^{Oc} = BR \cup \{(g_i, g_j): 1 \leq i \leq n, 1 \leq j \leq n \}$$
in place of $R^O$.  Still $R^{Oc}$ is Euclidean and it satisfies all other relations for eight deontic logics.
\end{proof}

From the above theorem, 

\begin{coro} 
For any of ten Smiley-Hanson systems, say $\bf SH$, and any formula $\phi$ of $\bf L_{1}$, $B$-translation is an 
embedding of $\bf L_{1}$ in $\bf SH$.
\end{coro}

%%%%%%%%%%%%%%%%%%%%

\section{A general theorem on the faithfulness}

First, as usual, a \it normal modal logic $X$  \rm is defined as follows:

\smallskip

(nor.1) $\bf K$ $\subseteq X$, 

(nor.2) $X$ is closed under modus ponens, substitution and the rule of necessitation (i.e., $\vdash \phi$ implies $\vdash \Box \phi$).

\smallskip

We may say that normal modal logics are extensions of $\bf K$.

In this section, we shall give a general theorem on the faithfulness of $B$-translation with respect to normal modal logics with transitive frames 
or irreflexive ones. 
\begin{theo}
Let $X$ be a normal modal logic.
Let $I$ be an index set such that $\omega \subseteq I$. Suppose that $X$ is complete with respect to a set of Kripke frames, 
say $F$ = $\{(G_\alpha, R_\alpha) : \alpha \in I\}$ where for any $\alpha \in I$, $G_\alpha$ is a nonempty infinite set and $R_\alpha$ 
is an accessibility relation on $G_\alpha$. Let $H$ be a set of Kripke frames such that

\smallskip

$(6.1.i)$ $H$ = $\{(G_i^*, R_i^*) : i \in \omega\}$,

$(6.1.ii)$ For any $i \in \omega$, 
$G_i^* = \{*, g_0, g_1, \dots , g_i\} \cup G_{\omega i},$

\smallskip 

\noindent where $G_{\omega i} = \{\eta_{ij}: j \in \omega\}$, $Card(G_{\omega i}) = \aleph_0$, 
$\{*, g_1, \dots , g_n\} \cap G_{\omega i} = \emptyset$ and $*, g_1, \dots , g_n$ are distinct and 
$R_i^* = \{(*, g_j): 0 \leq j \leq i\} \cup \{(\eta, \eta) : \eta \in G_{\omega} \}. $
If $H \subseteq F$ holds, then 
$B$-translation is faithful with respect to $X$.
\end{theo}

\begin{proof} Suppose the assumption and the condition of the theorem to be proved. 
Then we can take the same model-construction in the proof of $(\clubsuit)$ in \S4 
in order to prove the desired theorem. 
\end{proof}

\begin{defi} By $\bf ExtK_{tirr}$, we denote the set of all normal logics such that they are 
elements of $\bf ExtK$ and they are complete w.r.t. a set of transitive or irreflexive 
Kripke frames.
\end{defi}

\begin{coro} Let $X \in \mathbf{ExtK_{tirr}}$. $B$-translation is faithful 
with respect to $X$.
\end{coro}
\begin{proof}
From Theorem 6.1.
\end{proof}

\begin{coro} Let $X \in \mathbf{ExtK_{tirr}}$. $B$-translation is an embedding of $\bf L_1$ in $X$. 
\end{coro}
\begin{proof}
From Corollary 6.3.
\end{proof}

As a direct consequece of the above corollary, we have the following.

\begin{coro}
$B$-translation is an embedding of $\bf L_1$ in $\bf K4$.
\end{coro}
\begin{proof}
$\bf K4$ is stronger than $\bf K$. And it is complete with respect to transitive frames. Thus we obtain the soundness and 
faithfulness of $B$ from Corollary 6.4. 
\end{proof}

$\bf K4$ is a subsystem of the provability logic $\bf PrL$ (= $\bf K4$ $+$ $\Box (\Box \phi \supset \phi) \supset \Box \phi$). 
The subsystem is also called $\bf BML$. $\bf PrL$ is also characterized by 
$\bf K$ $+$ $\Box (\Box \phi \supset \phi) \supset \Box \phi$ 
(see \cite{cz}). 

We know that $\phi$ is a theorem of $\bf PrL$ iff $\phi$ is valid in all finite transitive and irreflexive frames 
(see $\rm e.g.$ Boolos \cite{boolos}; $\bf PrL$ is denoted by $\bf GL$ in \cite{boolos}, also see Carnielli and Pizzi \cite{carpi}). 
(Frames $(W, R)$ in which $W$ is finite and $R$ is irreflexive and transitive are called \it strict partial orders.\rm ) Thus we have:

\begin{coro} 
$B$-translation is an embedding of $\bf L_1$ in $\bf PrL$.
\end{coro}
\begin{proof}
In this case, we may take $G_{\omega} = \emptyset$ in the model construction of the proof of Therem 4.1. So we can do all in 
finite models. From Theorem 6.1. and Cororllary 6.4, our corollary holds, since our accessibility relations are all finte transitive and irreflexive. 
\end{proof}

%%%%%%%%%%%%%%%%%%%%%%%%%%

\section{Further general results for the faithfulness}

Considering the proof of Theorems 4.1, 5.3 and 6.1, we can obtain further general theorems for the faithfulness 
of $B$-translation and their applications.

\begin{theo}
Let $X$ be a normal modal logic. 
Let $I$ be an index set such that $\omega \subseteq I$. Suppose that $X$ is complete with respect to a set of Kripke frames, 
say $F$ = $\{(G_\alpha, R_\alpha) : \alpha \in I\}$, 

\noindent where for any $\alpha \in I$, $G_\alpha$ is a nonempty infinite set and $R_\alpha$ 
is an accessibility relation on $G_\alpha$. Let $H$ be a set of Kripke frames such that

\smallskip

$(7.1.i)$ $H$ = $\{(G_i^*, R_i^*) : i \in \omega\}$,

%\smallskip

$(7.1.ii)$ For any $i \in \omega$, 
$G_i^* = \{*, g_0, g_1, \dots , g_i\} \cup G_{\omega i},$
\noindent where $G_{\omega i} = \{\eta_{ij}: j \in \omega\}$, $Card(G_{\omega i}) = \aleph_0$, 
$\{*, g_1, \dots , g_n\} \cap G_{\omega i} = \emptyset$ and $*, g_1, \dots , g_n$ are distinct and 
$$R_i^* = \{(*, g_j): 0 \leq j \leq i\} \cup \{(g_j, g_k): j \neq k, 1 \leq j \leq i, 1 \leq k \leq i \}$$ $$ \cup \{(\eta, \eta) : \eta \in G_{\omega i} \}. $$
If $H \subseteq F$ holds, then 
$B$-translation is faithful with respect to $X$.
\end{theo}

\begin{proof} 
From Theorem 4.1, 5.3 and 6.1.
\end{proof}

Theorem 7.1 is applicable to normal modal logics with serial or irreflexive or Euclidean or almost symmetric Kripke frames.

\begin{theo}
Suppose that we take the same assumptions about $X$, $I$, $F$, $H$, $G_i^*$, $G_{\omega i}$ \rm (\it for any $i \in \omega$\rm) \it in Theorem 7.1 with  
$$R_i^* = \{(*, g_j): 0 \leq j \leq i\} \cup \{(g_j, g_j): 1 \leq j \leq i \} \cup \{(\eta, \eta) : \eta \in G_{\omega i} \}. $$
If $H \subseteq F$ holds, then 
$B$-translation is faithful with respect to $X$.
\end{theo}

\begin{proof} 
From Theorems 4.1, 5.3 and 6.1.
\end{proof}

Theorem 7.2 is applicable to normal modal logics with serial or transitive or irreflexive or almost reflexive or almost symmetric Kripke frames.

\begin{theo}
Suppose that we take the same assumptions about $X$, $I$, $F$, $H$, $G_i^*$, $G_{\omega i}$  \rm (\it for any $i \in \omega$\rm) \it in Theorem 7.1 with 
$R_i^* = \{(*, g_j): 0 \leq j \leq i\}$ 

\smallskip 

\noindent $ \cup \{(g_j, g_k): 1 \leq j, k \leq i, 1 \leq k \leq i   \} \cup \{(\eta, \eta) : \eta \in G_{\omega i} \}. $

\smallskip 

\noindent If $H \subseteq F$ holds, then 
$B$-translation is faithful with respect to $X$.
\end{theo}

\begin{proof} 
From Theorems 4.1, 5.3 and 6.1.
\end{proof}

Theorem 7.3 is applicable to normal modal logics with serial or transitive or irreflexive or Euclidean or almost reflexive or almost symmetric Kripke frames.

\begin{defi} By $\bf ExtK_{d}$, we denote the set of propositional modal logics stronger than $\bf K$ such that they are
complete w.r.t. a set of serial or transitive or irreflexive or Euclidean or almost reflexive or almost symmetric Kripke frames.
\end{defi}

\begin{theo} For any logic $X \in \mathbf{ExtK_{d}}$, $B$-translation is faithful 
with respect to $X$.
\end{theo}
\begin{proof}
From Theorem 7.3.
\end{proof}

\begin{coro} For any logic $X \in \mathbf{ExtK_{d}}$, $B$-translation is an embedding of $\bf L_1$ in $X$. 
\end{coro}
\begin{proof}
From Theorem 7.5.
\end{proof}

We recall the naming of modal logics as follows (refer to e.g. Poggiolesi \cite{pog} and Ono \cite{ono1}, 
also see Bull and Segerberg \cite{bulseg}):

\smallskip

$\bf KD$: $\bf K$ + $\Box \phi \supset \Diamond \phi$ ($\bf D$, serial relation)

$\bf K4$: $\bf K$ + $\Box \phi \supset \Box \Box \phi$ ($\bf 4$, transitive relation)

$\bf KD4$: $\bf K$ + $\bf D$ + $\bf 4$ (serial and transitive relation)

$\bf KB$: $\bf K$ + $\phi \supset \Box \Diamond \phi$ ($\bf B$, symmetric relation)

\smallskip

\noindent where $+$ means $\oplus$ in the sense of \cite{cz}, which the system is closed under modus ponens, substitution and 
the rule of necessitation. We use $+$ in the sense eveywhere in this paper. For those logics, we can give the corresponding general theorems. 

\begin{theo}
Suppose that we take the same assumptions about $X$, $I$, $F$, $H$, $G_i^*$, $G_{\omega i}$ \rm (\it for any $i \in \omega$\rm) \it in Theorem 7.1 with 
$$R_i^* = \{(*, g_j): 0 \leq j \leq i\} \cup \{(g_j, *): 0 \leq j \leq i\} \cup $$
$$\{(g_j, g_k): j \neq k, 1 \leq j \leq i, 1 \leq k \leq i \}$$ $$ \cup \{(\eta, \eta) : \eta \in G_{\omega i} \}. $$
If $H \subseteq F$ holds, then 
$B$-translation is faithful with respect to $X$.
\end{theo}
\begin{proof} 
From Theorem 4.1, 5.3 and 6.1, we can take the same model construction in the proof of Theorem 4.1 and prove it. 
We do not need take care of $\{(g_j, *): 0 \leq j \leq i\}$, since we may only consider the worlds which $*$ can access.
\end{proof}

Theorem 7.7 is applicable to normal modal logics with serial or irreflexive or Euclidean or symmetric Kripke frames.

\begin{theo}
Suppose that we take the same assumptions about $X$, $I$, $F$, $H$, $G_i^*$, $G_{\omega i}$ \rm (\it for any $i \in \omega$\rm) \it in Theorem 7.1 with 
$$R_i^* = \{(g_j, g_k): j \neq k, 1 \leq j \leq i, 1 \leq k \leq i \} \cup \{(*, *)\}.$$
If $H \subseteq F$ holds, then 
$B$-translation is faithful with respect to $X$.
\end{theo}
\begin{proof} 
From Theorem 4.1, 5.3 and 6.1, we can take the same model construction in the proof of Theorem 4.1 and verify the theorem.
What we need to check is that of the cases (Case 1) and (Case 2) with the formulas as such 
$$* \models \Box (p_{x_i} \supset p_{x_j}) \Leftrightarrow .  * \models p_{x_i} \supset p_{x_j} \wedge  
g_1 \models p_{x_i} \supset p_{x_j} \wedge  g_2 \models p_{x_i} \supset p_{x_j}$$
$$\wedge \cdots \wedge g_n \models p_{x_i} \supset p_{x_j}.$$
In (Case 1), $* \models p_{x_i} \supset p_{x_j}$ does not effect the proof of $* \not\models \Box (p_{x_i} \supset p_{x_j})$ in the proof of Theorem 4.1. 
In (Case 2),  $* \not\models p_{x_i} \supset p_{y_j}$ leads to  $* \not\models \Box (p_{x_i} \supset p_{y_j})$. 
From this we have  $* \not\models B(D_k)$. We remark that we shall take the same treatment as in the proof of Theorem 4.1, if the number of chain is 0.
\end{proof}

Theorem 7.8 is applicable to normal modal logics with serial or transitive or Euclidean Kripke frames.

\begin{theo}
Suppose that we take the same assumptions about $X$, $I$, $F$, $H$, $G_i^*$, $G_{\omega i}$ \rm (\it for any $i \in \omega$\rm) \it in Theorem 7.1 with 
$$R_i^* = \{(*, g_j): 0 \leq j \leq i\} \cup \{(g_j, *): 0 \leq j \leq i\} \cup \{(*, *)\}$$
If $H \subseteq F$ holds, then 
$B$-translation is faithful with respect to $X$.
\end{theo}
\begin{proof} 
The proof is similar to that of Theorem 7.8.
\end{proof}

Theorem 7.9 is applicable to normal modal logics with serial or symmetric Kripke frames.

From Theorems 7.8 and 7.9, we can obtain the following theorem.

\begin{theo}
Let $X$ be one of logics $\bf KD$, $\bf KD4$ and $\bf KB$. $B$-translation is faithful 
with respect to $X$.
\end{theo}

\begin{coro} 
Let $X$ be one of logics $\bf KD$, $\bf KD4$ and $\bf KB$. $B$-translation is an embedding of $\bf L_1$ in $X$. 
\end{coro}

\begin{proof}
From Theorem 7.10.
\end{proof}

\begin{theo}
Let $X$ be a normal modal logic. If $X$ has the finite model property, then 'infinite', $G_i^* = \{*, g_0, g_1, \dots , g_i\} \cup G_{\omega i}$ and 
$R_i^*$ can be replaced by 
'finite', $G_i^* = \{*, g_0, g_1, \dots , g_i\}$ and $R_i^* - \{(\eta, \eta) : \eta \in G_{\omega i} \}$ respectively, in Theorems 6.1, 7.1, 7.2, 7.3, 7.7, 7.8 and 7.9.
\end{theo}
\begin{proof}
Let $X$ be a normal modal logic. Suppose that $X$ has the finite model property. 
In the proof of Theorem 4.1, we can take $G_{\omega} = \emptyset$. Then all proofs of Theorems 6.1, 7.1, 7.2, 7.3, 7.7, 7.8 and 7.9 
follow the modification without any trouble.
\end{proof}

For the finite model property, refer to e.g. Carnielli and Pizzi \cite{carpi}, Ono \cite{ono1} (especially for algebraic models) and so on.

%%%

\section{Some open problems and conjectures}

I shall give several open problems below. 

\bf Open problem 1\rm . Is the condition of Theorems 6.1, 7.1, 7.2, 7.3, 7.7 and 7.8 ($\rm i.e.$, as \{(6.1.i) and (6.1.ii))\} and \{(7.1.i) and (7.1.ii))\}) 
the best possible one for the faithfulness of $B$-translation ?

\bf Open problem 2\rm . Give an algebraic proof of the faithfulness of $B$-translation.

\bf Open problem 3\rm . Search for possibilities of embedding Le\'{s}niewski's propositional ontology $\bf L_1$ in 
bi-intuitionistic logic and the display logic for bi-intuitionistic logic (see Gor\'{e} \cite{gore1} and Wansing \cite{wan1} 
for those logics).

\bf Open problem 4\rm . Search for possibilities of embedding Le\'{s}niewski's propositional ontology $\bf L_1$ in logics 
treated in Marx and Venema \cite{marx-venema}.

I think that this book is very important to consider. It is also worth considering Gabbay, Kurucz, Wolter and Zakharyaschev 
\cite{gab-many} for temporal and computational aspects, and products for modal logics, for example.

In relation to the theme of this paper, I shall give my conjectures.

\bf Conjecture 1\rm . $\bf L_1$ is embedded in some bimodal (multi) logic by some sound translation.

\bf Conjecture 2\rm . $\bf L_1$ is embedded in many other modal logics which are listed in, for example, 
Chagrov and Zakharyaschev \cite{cz}, Humberstone \cite{humber} and  Gabbay and Guenthner \cite{gab2}.

\bf Conjecture 3\rm . $\bf L_1$ is embedded in some temporal logics with several ways. 

\bf Conjecture 4 [Fundamental Conjecture]\rm . Le\'{s}niewski's elementary ontology $\bf L$ is embedded in some \it first-order 
\rm modal predicate logic.

As mentioned in the Introduction of this paper, this conjecture 4 is the most important open problem in the next stage of our direction of study. 
I believe that by introuducing a modal operator, we could reduce second-order to first-order for the embedding. 
In a similar direction, a certain attempt has been carried out in Urbaniak \cite{urbaniak-Thesis} and \cite{urbaniak-plural}. 
If we find such an embedding in a particular formal system with a modal operator of a definite strength, it would be another breakthrough in the study of 
Le\'{s}niewski's systems.

%%%%%%  ACKNOWLEDGEMENTS SECTION
\paragraph{Acknowledgements.} 
I would like to thank Prof. Andreas Blass, Prof. Emeritus Mitio Takano, the late Prof. Stanis\/{l}aw \'{S}wierczkowski and the late Prof. Emeritus Arata Ishimoto 
for their valuable comments, stimulation and encouragement.  I would also like to thank, especially, the late Prof. Anne Troelstra and 
Prof. Dirk van Dalen for their encouragement. Further I appreciate anonymous referees for their comments to make this paper better.

%%%%%%  THE BIBLIOGRAPHY
%%%%%%  See the examples and format yours according to them

%%%%%%  AUTHOR'S ADDRESS INFORMATION AT THE END OF THE PAPER:

\AuthorAdressEmail{Takao Inou\'{e}}{Department of Applied Informatics \newline Faculty of Science and Engineering\\
Hosei University\\
Kajino-cho 3-7-2\\
Koganei-shi, Tokyo, Japan}{takao.inoue.22@hosei.ac.jp \newline takaoapple@gmail.com}

\end{document}